\documentclass[11pt]{amsart}

\usepackage{eucal,amsrefs,graphicx,amssymb,mathrsfs}
\usepackage[all, knot]{xy}

\theoremstyle{definition}
\newtheorem{ex}{Example}[section]
\newtheorem*{ex*}{Example}

\newtheorem*{DMM*}{Dynamical Manin-Mumford Problem}
\newtheorem*{question*}{Question}

\theoremstyle{plain}
\newtheorem{thm}{Theorem}[section]
\newtheorem{prop}[thm]{Proposition}
\newtheorem*{cor*}{Corollary}
\newtheorem{cor}[thm]{Corollary}
\newtheorem{lem}[thm]{Lemma}
\newtheorem*{claim}{Claim}
\newtheorem*{thm*}{Theorem}
\newtheorem{conj}{Conjecture}[section]
\newtheorem*{thmA}{Theorem A}
\newtheorem*{thmB}{Theorem B}
\newtheorem*{thmC}{Theorem C}

\theoremstyle{remark}
\newtheorem*{remark}{Remark}

\numberwithin{equation}{section}

\renewcommand{\div}{\mathrm{div}}

\newcommand{\vect}[1]{\mathbf{#1}}

\DeclareMathOperator{\Spec}{Spec}
\DeclareMathOperator{\supp}{supp}
\DeclareMathOperator{\rank}{rank}

\DeclareMathOperator{\tor}{tor}

\DeclareMathOperator{\Hom}{Hom}

\DeclareMathOperator{\Span}{span}
\DeclareMathOperator{\Star}{Star}

\DeclareMathOperator{\ad}{ad}

\DeclareMathOperator{\row}{Row}

\DeclareMathOperator{\Mat}{Mat}
\DeclareMathOperator{\PrePer}{PrePer}

\def\C{\mathbb{C}}
\def\Q{\mathbb{Q}}
\def\P{\mathbb{P}}
\def\R{\mathbb{R}}

\def\Z{\mathbb{Z}}
\def\P{\mathbb{P}}
\def\PP{\mathcal{P}}
\def\K{{\mathbb{K}}}
\def\N{\mathbb{N}}
\def\T{\mathbb{T}}
\def\G{\mathbb{G}}

\def\v{\mathbf{v}}

\def\cL{\mathcal{L}}

\def\m{\mathrm{m}}
\def\bb{\vect{b}}
\def\aa{\vect{a}}
\def\xx{P}
\def\bmu{{\pmb{\mu}}}

\def\Qbar{\bar{\Q}}

\def\fan{\Delta}

\def\n0{{\bf n_0}}
\def\one{\mathbf{1}}

\def\Homag{\Hom_{\rm{alg.gp.}}}

\providecommand{\norm}[1]{\lVert#1\rVert}

\def\transp #1{\vphantom{#1}^{\mathrm t}\! {#1}}

\def\maxp{\mathrm{max}^+}

\begin{document}

\title[Arithmetic Dynamics of Monomial Maps]
{On the Arithmetic Dynamics of Monomial Maps}

\author{Jan-Li Lin}

\address{Mathematics Department, Northwestern University, Evanston \\ IL 60208 \\ USA}

\email{janlin@math.northwestern.edu}



\begin{abstract}
We generalized several results for the arithmetic dynamics of monomial maps, including Silverman's conjectures on height growth, dynamical Mordell-Lang conjecture, and dynamical Manin-Mumford conjecture. These results are originally known for monomial maps on algebraic tori. We extend the results to arbitrary toric varieties.
\end{abstract}

\maketitle


\section{Introduction}

Suppose that $X$ is an $n$-dimensional complex normal projective variety and $\varphi:X\dashrightarrow X$ is a dominant rational selfmap on $X$, defined over $\Qbar$. Given an ample divisor $D$ of $X$, one define the degree of $\varphi$ with respect to $D$ as $\deg_D(\varphi):=\varphi^*D\cdot D^{n-1}$. The asymptotic growth of the degree under iterates of $\varphi$ is measured by the {\em dynamical degree} 
\[
\delta(\varphi):=\lim_{k\to\infty} \deg_D(\varphi^k)^{1/k}.
\]
The dynamical degree is independent of the choice of $D$ (see \cite[Theorem 1]{Dang}) and is invariant under birational conjugation.

From the arithmetic point of view, using the Weil height machine, we have the height function $h_{X,D}$ on $X(\Qbar)$ associated to $D$. Let $X_\varphi(\Qbar)$ denote the set of points $P\in X(\Qbar)$ whose forward $\varphi$-orbit is well-defined. For $P\in X_\varphi(\Qbar)$, one can define the {\em arithmetic degree} of $\varphi$ at $P$ as
\[
\alpha_\varphi(P)=\lim_{k\to\infty}h_{X,D}^+(\varphi^k(P))^{1/k}
\]
provided that the limit exists, where we set $h_{X,D}^+(P):=\max\{h_{X,D}(P),1\}$. The arithmetic degree is also independent of the choice of $D$ (see \cite[Remark 39]{Sil}).

The relationship between $\delta(\varphi)$ and $\alpha_\varphi(P)$ is an interesting topic. In fact, in the papers \cite{KaSi, Sil}, the authors made several conjectures, as follows.

\begin{conj}
\label{main_conjectures}
Let $P\in X_\varphi(\Qbar)$.
\begin{enumerate}
\item The limit defining $\alpha_\varphi(P)$ exists. 
\item $\alpha_\varphi(P)$ is an algebraic integer.
\item The set of arithmetic degrees $\{\alpha_\varphi(Q)\ |\ Q\in X_\varphi(\Qbar)\}$ is finite.
\item If the forward $\varphi$-orbit of $P$ is Zariski dense in $X$, then $\alpha_\varphi(P)=\delta(\varphi)$.
\end{enumerate}
\end{conj}

In the paper \cite{Sil}, Silverman proved the conjectures for monomial maps on an algebraic torus. 
Recall that, given an $n\times n$ integer matrix $A=(a_{ij})$, the monomial (self)map associated to $A$ is a is the map $\varphi_A: (\C^*)^n \to (\C^*)^n$ defined by 
\[
\varphi_A(x_1,\cdots,x_n) =
            (x_1^{a_{11}} x_2^{a_{12}} \cdots x_n^{a_{1n}},\cdots,
            x_1^{a_{n1}} x_2^{a_{n2}} \cdots x_n^{a_{nn}}).
\]
$\varphi_A$ is a group endomorphism when we identify $(\C^*)^n$ with the multiplicative algebraic group $\G_\m^n(\C)$. It is surjective if $\det(A)\ne 0$. Moreover, $\varphi_A$ extends to an equivariant rational selfmap on any toric varieties $X$ which contain $(\C^*)^n$ as a dense open subset. A goal of this paper is to extend Silverman's results to toric varieties. More precisely, we prove the following.

\begin{thmA}
\label{thm:main_1}
Let $\varphi=\varphi_A:X\dashrightarrow X$ be a monomial map defined on a toric variety $X=X(\fan)$ by an $n\times n$ matrix $A$ with $\det(A)\ne 0$. Suppose that $f_1(t),\cdots, f_s(t)\in\Z[t]$ be the monic irreducible factors of the characteristic polynomial $\det(A-t I)$ of $A$. Then $\alpha_\varphi(P)$ exist for all $P\in X(\Qbar)_\varphi$, and
\[
\{\alpha_\varphi(P)\ |\ P\in X(\Qbar)_\varphi\} = \{ 1, \rho(f_1),\cdots,\rho(f_s)\},
\]
where for a polynomial $f(t)\in\C[t]$, 
$\rho(f):=\max\{\,|\alpha|\, :\, \alpha\text{ is a root of $f$}\}$.
\end{thmA}

Notice that Theorem~\ref{thm:main_1} implies parts (1), (2) and (3) of Conjecture~\ref{main_conjectures}. Part (4) is a consequence of the description of orbits under monomial maps in Section 2.7. Also noticed that Silverman proved the result for $P\in\G^n_\m(\Qbar)$ in \cite[Corollary 32]{Sil}. The improvement of the current paper is to extend the result to points (with well-defined $\varphi$-orbits) on an arbitrary toric variety. This extension contains some meaningful cases. For example, monomial polynomial maps defined on $\C^n$, not just $(\C^*)^n$, hence for these maps, such generalization is natural and even necessary.

The main tool we use to prove the theorem is the orbit-cone correspondence for toric variety and a careful analysis of how the orbits behave under monomial maps. As another consequence, we can also extend the dynamical Mordell-Lang conjecture to monomial maps on arbitrary projective toric varieties. The dynamical Mordell-Lang conjecture predicts, for an endomorphism $\varphi$ on a quasiprojective variety $X$, how the $\varphi$-orbit of a point $P\in X$ intersects a subvariety $Y\subset X$. More precisely, the conjectures states that the set $\{ k\in \N\ |\ \varphi^k(P)\in Y\}$ is a finite union of arithmetic progressions. The conjecture is proven for many classes of maps, e.g., \'{e}tale maps on quasiprojective varieties~\cite{BGT}, plane polynomial maps~\cite{Xie}. See the book \cite{BGT_book} and the references therein for more information.

\begin{thmB}
Let $X=X(\fan)$ be a toric variety, $\varphi=\varphi_A$ be a monomial map, and $Y\subset X$ be a closed subvariety of $X$. Then for all $P\in X_\varphi$, the sequence
\[
\{ k\in \N\ |\ \varphi^k(P)\in Y\}
\]
 is a finite union of arithmetic progressions.
\end{thmB}

Another goal of this paper is to describe the set of preperiodic points for monomial maps on toric varieties. This is also done in \cite{Sil} for $\G^n_\m$ for a major case. We first generalize to all monomial maps for $\G^n_\m$, then again we extend to toric varieties. An important motivation to study the preperiodic points is the dynamical Manin-Mumford problem.

\begin{DMM*}
Suppose that $X$ is a quasiprojective variety and $\varphi:X\to X$ is a dominant endomorphism. Describe all positive-dimensional irreducible subvarieties $Y\subset X$ such that the set of preperiodic points of $\varphi$ contained in $Y$ is Zariski dense in $Y$.
\end{DMM*}

There are recent results for this problem for plane polynomial automorphisms~\cite{DF} and for split rational maps~\cite{GNY}.
To answer the problem for monomial maps, we obtain the following characterization for closed subvarieties of toric varieties containing Zariski dense set of preperiodic points for a monomial map.  
For notations and other details, see Section~\ref{section:DMM}.

\begin{thmC}
Let $\varphi:X \dashrightarrow X$ be a monomial selfmap on a toric variety $X$. If we write $X_\varphi$ as a disjoint union of $\T$-orbits $X_\varphi = \bigsqcup O_\sigma$, and write the set of preperiodic points of $\varphi$ as $\PrePer(\varphi)=\bigsqcup G_\sigma^\div$ for connected algebraic subgroups $G_\sigma$ of $O_\sigma$ (see Proposition~\ref{prop:preper_points_mono_maps} for the description of $\PrePer(\varphi)$). 

For a closed irreducible subvariety $Y\subset X_\varphi$, let $\sigma$ be the unique cone such that $Y\cap O_\sigma$ is Zariski dense in $Y$, and let $\pi: O_\sigma\to O_\sigma / G_\sigma$ be the quotient map. Then the preperiodic points of $\varphi$ is Zariski dense in $Y$ if and only if $\pi(Y\cap O_\sigma)$ is a torsion translate of an algebraic subgroup of $O_\sigma$.
\end{thmC}

The rest of the paper is as follows. In Section~\ref{section:background}, we recall results about algebraic tori and toric varieties, and develop necessary background for the later sections. The dynamical Mordell-Lang theorem is discussed in Section~\ref{section:DML}. Height growth is studied in Section~\ref{section:height_growth}. Finally, we describe the preperiodic points of monomial maps and the dynamical Manin-Mumford problem in Section~\ref{section:preperiodic_points}.


\section{Algebraic Tori, Toric Varieties and Monomial Maps}
\label{section:background}

In this section we recap the basic definitions, notations and properties of toric varieties. For more details see the standard text by Fulton~\cite{Fulton}. 

\subsection{Algebraic tori}

Let $\T\cong\G_\m^n(\C)$ be the algebraic torus of dimension $n$. There are two lattices associated with $\T$: the lattice of one-parameter subgroups $N:=\Homag(\C^*,\T)$, and the lattice of characters $M:=\Homag(\T,\C^*)$. Both $M$ and $N$ are isomorphic to $\Z^n$, as follows. 
For $\aa=(a_1,\cdots,a_n)\in \Z^n$, the one-parameter subgroup associated to $\aa$, 
denoted by $\lambda_\aa:\C^*\to\T$, is given by 
$\lambda_\aa(x)= (x^{a_1},\cdots,x^{a_n})$.
For $\bb=(b_1,\cdots,b_n)\in \Z^n$, the character associated to $\bb$,
denoted by $\chi^\bb:\T\to\C^*$, is given by
$\chi^\bb(x_1,\cdots,x_n) = x_1^{b_1}\cdots x_n^{b_n}$.
There is a natural pairing between $N$ and $M$ which makes them dual to each other, i.e., $M\cong N^\vee$ and $N\cong M^\vee$. An alternative way to describe $\T$ is as $\T=\Spec \C[M]$, where $\C[M]$ denotes the group ring associated to $M$.

We introduce the following convenient multi-index notations. For $x\in\C^*$ and $\aa=(a_1,\cdots,a_n)\in \Z^n$, we denote $x^\aa:= (x^{a_1},\cdots,x^{a_n})$. For $\xx=(x_1,\cdots,x_n)\in\T$ and $\bb=(b_1,\cdots,b_n)\in \Z^n$, we denote $\xx^\bb := x_1^{b_1}\cdots x_n^{b_n}$.
Recall that, from the group structure on $\T$, we also have $\xx^k = (x_1^k,\cdots, x_n^k)$ for $k\in\Z$ . Thus, we have the equations $\xx^{k\bb}=(\xx^k)^\bb=(\xx^\bb)^k$.

\subsection{Algebraic subgroups of a torus}


A sublattice $L\subset M$ is {\em saturated} if $L=\Span_\R(L)\cap M$. A saturated sublattice $L\subset M$ corresponds to an subtorus $\G_L$ of $\T$, which is defined by equations
\[
\G_L=\bigcap_{\aa\in L} \left\{\xx\in\T\ |\ \xx^\aa=\one\right\}.
\]

Alternatively, pick a set of basis of $L$, say $\aa_1,\cdots,\aa_r$, where $r$ is the rank of $L$. One can check that $\G_L:=\bigcap_{i=1}^r \left\{\xx^{\aa_i}=\one\right\}$.
Thus, the above finite intersection is independent of the choice of basis. The dimension of $\G_L$ is $n-r$ (i.e. the codimension of $\G_L$ is $r$).

For an arbitrary (not necessarily saturated) sublattice $K$ of $M$, we can also define $\G_L=\bigcap_{\aa\in L} \left\{\xx^\aa=\one\right\}$. It is an algebraic subgroup of $\T$, but not a subtorus if $L$ is not saturated. 
If $L$ is not saturated, then the group $\G_L$ is not connected. The connected component of $\G_L$ containing the unit $\one$ is $\G_{\widetilde{L}}$, where $\widetilde{L}$ is the {\em saturation} of $L$ in $M$, defined by $\widetilde{L}:=\Span_\R(L)\cap M$. The number of connected components in $\G_L$ is $[\widetilde{L}:L]$.

\subsection{Divisible hull}
Let $G$ be an algebraic subgroup of $\T$, the {\em divisible hull} of $G$, denoted by $G^{\div}$, is defined as
\[
G^{\div} := \{ \xx\in\T\ |\ \xx^k\in G 
\text{ for some $k\ge 1$}\}. 
\]

\begin{prop} Let $L$ be a sublattice of $M$, $\T_{\tor}$ be the subgroup of torsion points of $\T$, i.e. $\bmu\in\T_{\tor}$ if and only if each coordinate of $\bmu$ is a root of unity. 
\begin{enumerate}
\item $G^{\div} = G \cdot \T_{\tor} $,

\item If $L\subseteq L'$ are two sublattices of $M$, then $\G_{L'}\subseteq \G_L$ and $\G_{L'}^{\div}\subseteq\G_{L}^{\div}$.

\item Suppose that $\{\aa_1,\cdots, \aa_r\}$ is a basis of $L$, then 
\[
\G_L^{\div} = \bigcup_{\bmu_1,\cdots,\bmu_r\in\T_{\tor}} 
              \limits{\bigcap}_{i=1}^r \left\{\xx^{\aa_i}=\bmu_i\right\}.
\]

\item Suppose $L, L'$ are two sublattices of $M$ such that $\Span_\R(L)=\Span_\R(L')$, then $\G_L^{\div}=\G_{L'}^{\div}$. 

\item The set of all sublattices of $M$ with the same $\R$-span (or, equivalently, with the same saturation) forms a directed system.   

\item Let $V=\Span_\R(L)$ and $\cL$ be the directed system of all sublattices of $M$ whose $\R$-span is $V$, then $\G_L^{\div}= \displaystyle{\varprojlim_{\cL} \G_{L'} = \cup_{L'\in\cL} \G_{L'}}$.

\end{enumerate}
\end{prop}

\begin{proof}
(1) is proved in Silverman's paper~\cite{Sil}. Part (2) is directly by the definition of $\G_L$. For (3), recall that $\G_L=\bigcap_{i=1}^r \left\{\xx^{\aa_i}=\one\right\}$. Thus, for $\xx\in\G_L^{\div}$, the condition $\xx^k\in\G_L$ means $(\xx^k)^{\aa_i}=(\xx^{\aa_i})^{k}=\one$ for all $i$, which is equivalent to that each coordinate of $\xx^{\aa_i}$ is a $k$-th root of unity. This proves (3).

For (4), since $\Span_\R(L)=\Span_\R(L')$ implies that $L$ and $L'$ have the same saturation, it suffices to show that $\G_L^{\div}=\G_{\widetilde{L}}^{\div}$. We have $\G_{\widetilde{L}}^{\div}\subseteq\G_L^{\div}$ by (2). For the other inclusion, notice that $\aa\in \widetilde{L}$ implies $k\aa\in L$ for some integer $k>0$. Thus for $\xx\in\G_{L}$, we have $\xx^{k\aa}=(\xx^k)^\aa=\one$. Apply this to basis vectors $\{\aa_1,\cdots,\aa_r\}$ of $\widetilde{L}$ and let $k$ be the product (or lcm) of all the $k_i$'s which work for $\aa_i$, we conclude that $\xx^k\in \G_{\widetilde{L}}$. This shows $\G_{L}\subseteq\G_{\widetilde{L}}^{\div}$ which then implies that $\G_{L}^{\div}\subseteq\G_{\widetilde{L}}^{\div}$.

To prove (5), we need to show the following: suppose $L, L'$ are two sublattices of $M$ such that $\Span_\R(L)=\Span_\R(L')$, then $L\cap L'$ is also a sublattice of $M$ with the same $\R$-span. This follows from the fact that there exist positive integers $k$ and $k'$ such that $kL\subset L'$ and $k'L'\subset L$. 

Finally, for (6), since all lattices in $\cL$ have the same divisible hull by (4), it suffices to show that if $\xx\in \G_L^{\div}$, then $\xx\in\G_{L'}$ for some $L'\in\cL$. Notice that $\xx\in \G_L^{\div}$ means $\xx^k\in \G_L$ for some $k>0$, i.e. $\xx^{k\aa}=\one$ for all $\aa\in L$. This then implies that $\xx\in \G_{kL}$ where $kL\in\cL$. 
\end{proof}

\subsection{Toric varieties}

A toric variety is a normal variety $X$ which contains an algebraic torus $\T$ as a dense open subset, and an action $\T\times X\to X$ which extends the action of $\T$ on itself by multiplication. Toric varieties can be constructed from a {\em fan} $\fan$ in $N_\R:= N\otimes_\Z \R$. 
A fan is a collection of cones in $N_\R$ such that: (1) suppose $\sigma\in\fan$ and $\tau$ is a face of $\sigma$, then $\tau\in\fan$; and (2) if $\sigma,\sigma'\in\fan$, then $ \sigma\cap \sigma'\in\fan$.

For each cone $\sigma\in\fan$, one defines an affine toric variety $U_\sigma := \Spec \C[\sigma^{\vee}\cap M]$.
These affine toric varieties glue along their intersections to form the toric variety $X(\fan)$. The support of a fan $\fan$ is the subset $\supp(\fan):=\cup_{\sigma\in\fan}\sigma\subset N_\R$. The toric variety $X(\fan)$ is complete if and only if $\supp(\fan)=N_\R$. Complete toric varieties are not always projective. Nevertheless, by toric Chow's lemma, for any complete toric variety $X$, one can always make a $\T$-equivariant proper modification to obtain a projective toric variety $\widetilde{X}$.

\subsection{Orbit decomposition and the orbit-cone correspondence}
Let $X=X(\fan)$ denote the toric varieties associated to the fan $\fan$. We can write $X$ as a disjoint union of orbits under the torus action. More precisely, for each $\tau\in\fan$, pick any point $\aa$ in the {\em relative interior}\footnote{The relative interior of a cone $\sigma\subset N_\R$ is defined as the interior of $\sigma$ in $\Span_\R(\sigma)$.} $\tau^\circ$ of $\tau$, there is a corresponding distinguished point $P_\tau := \lim_{x\to 0} \lambda_\aa(x)$. One can show that the limit always exists in $X$, and is independent of the point $\aa$.
The orbit $O_\tau$ is then the orbit of $P_\tau$ under the $\T$-action. An alternative (and abstract) way to describe $O_\tau$ is as $O_\tau=\Spec \C[\tau^\perp\cap M]\subset X$. Each $O_\tau$ is also a torus, i.e., $O_\tau\cong\G_m^{n-k}(\C)$ where $k=\dim_\R(\tau)$. 

The closure of $O_\tau$ in $X$ is denoted by $V(\tau)$. The correspondence $\tau\mapsto V(\tau)$ is order-reversing in the sense that $\tau\subseteq \tau'$ if and only if $V(\tau)\supseteq V(\tau')$; and since $O_\tau$ and $V(\tau)$ have the same dimension, we have $\dim_\C(V(\tau))=n-\dim_\R(\tau)$.
Each $V(\tau)$ is itself a toric variety, as described in the following. Let $N_\tau := \Span_\R(\tau)\cap N$ denote the sublattice of $N$ containing all lattice points in the vector subspace spanned by $\tau$. Define $N(\tau)=N/N_\tau$ and $M(\tau)=\tau^\perp \cap M$. The {\em star} of $\tau$ in $\fan$, denoted $\Star(\tau)$, is a fan in $N(\tau)_\R$. Under these notations, we have $V(\tau)\cong X(\Star(\tau))$. The definition of $\Star(\tau)$ and the concrete isomorphism between $V(\tau)$ and $X(\Star(\tau))$ are described in~\cite[p.53]{Fulton}.

\subsection{Monomial maps}
Let $\Mat_{m,n}(\Z)$ denote the set of $m \times n$ integer matrices. For a matrix $A=(a_{ij})\in\Mat_{m,n}(\Z)$, the associated monomial map is defined by
\begin{align*}
\varphi_A:  \G_\m^n(\C)        & \longrightarrow \G_\m^m(\C) \\
            (x_1,\cdots,x_n)& \longmapsto 
            (x_1^{a_{11}} x_2^{a_{12}} \cdots x_n^{a_{1n}},\cdots,
            x_1^{a_{m1}} x_2^{a_{m2}} \cdots x_n^{a_{mn}})
\end{align*}
The map $\varphi_A$ is a homomorphism of the algebraic groups $\G_\m^n(\C) \to \G_\m^m(\C)$ and thus extends to a rational map from any $n$-dimensional toric variety to any $m$-dimensional toric variety. We still denote the rational map by $\varphi_A$ to avoid too many notations. The rational map $\varphi_A$ is equivariant under the torus action. Therefore, an orbit (under the torus action) either maps to another orbit, or the whole orbit is in the indeterminacy set. We can read the behavior of an orbit under monomial maps from the action on the cones in the fan associated to the toric variety. We will describe this in the next subsection.

\subsection{Orbits under monomial maps}
\label{section:orbits_under_monomial_maps}
The monomial map $\varphi_A$ acts on the groups of one parameter subgroups and the characters by post-composing and pre-composing with $\varphi_A$, respectively, as follows. 
\begin{align*}
{\varphi_A}_\#:N &\longrightarrow N' & \varphi_A^\#&:M'\longrightarrow M \\
             \lambda_\aa & \longmapsto \varphi_A\circ\lambda_\aa
             & \chi^\bb  & \longmapsto \chi^\bb\circ\varphi_A
\end{align*}
After identifying $N$ and $M$ with $\Z^n$, $N'$ and $M'$ with $\Z^m$, ${\varphi_A}_\#$ and $\varphi_A^\#$ are integral linear transformations: ${\varphi_A}_\#$ is multiplying by the matrix $A$ and $\varphi_A^\#$ by $\transp{A}$.
For this reason, from now on we will denote the action by $\varphi_A$ on $N$ and $M$ by $A$ and $\transp{A}$ instead of ${\varphi_A}_\#$ and $\varphi_A^\#$, respectively.

Suppose $\fan$ and $\fan'$ are fans in $N_\R$ and $N'_\R$, respectively. The map $A$ acts on the cones in $\fan$, and the action reflects how $\varphi_A$ acts on the corresponding orbits. 

\begin{lem}
Let $\sigma\in\fan$. 
\begin{enumerate}
\item If there is a $\sigma'\in\fan'$ such that $A(\sigma)\subset\sigma'$, then $\varphi_A$ is well-defined for all points in $O_\sigma$ and $\varphi_A(O_\sigma)\subset O_{\sigma'}$.
\item If there is no $\sigma'\in\fan'$ such that $A(\sigma)\subset\sigma'$, then every point in the orbit $O_\sigma$ is in the inderterminacy set, i.e. $O_\sigma\subset I(\varphi_A)$.
\end{enumerate}
\end{lem}
\begin{proof}
For (1), one can see this by looking at the distinguished point $P_\sigma$. Since $A(\aa)\in (\sigma')^\circ$ for $\aa\in\sigma^\circ$, this implies $\varphi_A(P_\sigma)=P_{\sigma'}$ and thus $\varphi_A(O_\sigma)\subset O_{\sigma'}$.

On the other hand, the condition in (2) implies that there are $\aa_1, \aa_2\in \sigma^\circ\cap N$ and $\sigma'_1,\sigma'_2\in\fan'$ such that $A(\aa_i)\in (\sigma'_i)^\circ$, $i=1,2$. This then implies that 
\[
\varphi_A: P_\sigma \longmapsto \varphi_A\left(\lim_{x\to 0} \lambda_{\aa_i}(x)\right)=P_{\sigma'_i}.
\]
That is, if we approach the distinguished point $P_\sigma$ using different one-parameter subgroups, and then apply the monomial map $\varphi_A$, we in fact get different results. Thus, the image $\varphi_A(P_\sigma)$ is not well-defined, and $P_\sigma\in I(\varphi_A)$. Moreover, since $\T P_\sigma=O_\sigma$ and the action is transitive, we conclude that $O_\sigma\subset I(\varphi_A)$.
\end{proof}

Notice that
\[
\fan_A = \{ \sigma\in\fan\ |\ \exists \sigma'\in\fan'\text{ such that } A(\sigma)\subset\sigma' \}
\]
is a fan. Indeed, if $\tau$ is a face of $\sigma$ and $A(\sigma)\subset\sigma'$, then $A(\tau)\subset A(\sigma)\subset\sigma'$. Also, if $A(\sigma)\subset\sigma'$ and $A(\tau)\subset\tau'$, then $A(\sigma\cap\tau)\subset\sigma'\cap\tau'\in\fan'$. Therefore, $X(\fan_A)$ is still a toric variety. 

When we have $\varphi_A(O_\sigma)\subset O_\sigma'$, then the restriction 
\[
\varphi_A|_{O_\sigma}:O_\sigma\to O_{\sigma'}
\]
is again a monomial map. Notice that $A(\sigma)\subset\sigma'$ implies $A(N_\sigma)\subset N'_{\sigma'}$. Thus $A$ induces a map on the quotient spaces
\[
\bar{A}:N(\sigma)\to N'(\sigma').
\]
The monomial map from $O_{\sigma}$ to $O_{\sigma'}$ is the map $\varphi_{\bar{A}}$ associated with the map $\bar{A}$. Moreover, it extends to a (rational) monomial maps between the orbit closures (which are toric varieties)
\[
\varphi_{\bar{A}}:V(\sigma)\dashrightarrow V(\sigma').
\]

Next, let $\Mat^+_n(\Z)$ be the set of $n\times n$ integer matrices with nonzero determinants, a matrix $A\in\Mat^+_n(\Z)$ induces a rational monomial selfmap $\varphi_A$ on any $n$-dimensional toric variety $X(\fan)$. 
From the above discussion, we can summarize how the torus orbits act under the monomial map $\varphi_A: X(\fan)\dashrightarrow X(\fan)$ using a directed graph, as follows.
\begin{itemize}
\item Vertices are labeled by the cones in the fan $\fan$, with an extra vertex labeled `I' (for indeterminacy). 
\item For a cone $\sigma\in\fan$, if there is another $\sigma'\in\fan$ such that $A(\sigma)\subset\sigma'$, then we draw a (directed) edge from $\sigma$ to $\sigma'$ to denote that $\varphi_A$ maps $O_\sigma$ to $O_{\sigma'}$. 
\item If such $\sigma'$ does not exist in $\fan$, then we draw an edge from $\sigma$ to `I' to denote that the whole orbit $O_\sigma$ is indeterminant. 
\end{itemize}

Recall that $X(\fan)_{\varphi_A}$ denotes the set of points $P\in X(\fan)$ that has well-defined orbit under the map $\varphi_A$, i.e. $\varphi_A^k(P)$ is well-defined for $k\ge 0$. From the graph, we know that $X(\fan)_{\varphi_A}$ is the union of the orbits $O_\sigma$ such that $\sigma$ never reach `I' along the edges in the graph. 

One can show that $X(\fan)_{\varphi_A}$ is a toric variety. In fact, let
\[
\widetilde{\fan}_A=\bigcap_{k=1}^{|\fan|} \fan_{A^k},
\]
then $\widetilde{\fan}_A$ is a fan and $X(\fan)_{\varphi_A}=X(\widetilde{\fan}_A)$ is still a toric variety. This is worth mentioning since in general, for a rational map $\varphi:X\dashrightarrow X$, the space $X_\varphi$ is not necessarily a variety.

\begin{ex}
The following picture shows a fan $\fan$ with the corresponding toric variety $X(\fan)\cong \P^2$.

\includegraphics[scale=1]{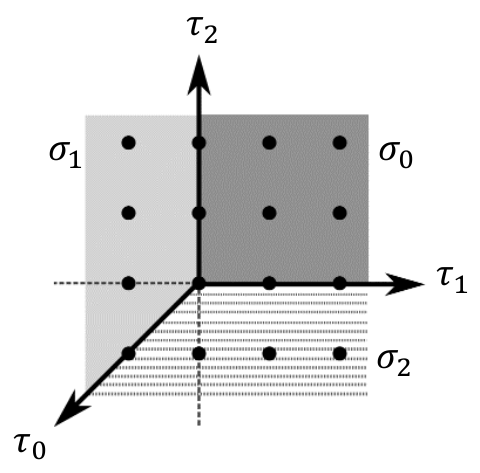}

More precisely, let $\tau_0, \tau_1$, and $\tau_2$ be the one-dimensional cones generated by 
$\left(\begin{smallmatrix} -1  \\ -1 \end{smallmatrix}\right),
\left(\begin{smallmatrix} 1  \\ 0 \end{smallmatrix}\right)$, 
and 
$\left(\begin{smallmatrix} 0  \\ 1 \end{smallmatrix}\right)$, respectively. Let $\sigma_0$ be the cone generated by 
$\left(\begin{smallmatrix} 1  \\ 0 \end{smallmatrix}\right)$ 
and 
$\left(\begin{smallmatrix} 0  \\ 1 \end{smallmatrix}\right)$; $\sigma_1$ and $\sigma_2$ are defined similarly. Then the fan $\fan$ is
\[
\fan=\Bigl\{ \{{\bf 0}\}, \tau_0, \tau_1, \tau_2, \sigma_0, \sigma_1, \sigma_2 \Bigr\}.
\]

Suppose that the homogeneous coordinates of $\P^2$, written as $[z_0:z_1:z_2]\in\P^2$, are defined such that $V(\tau_i)=(z_i=0)$, $i=0,1,2$. Then $V(\sigma_0)=[1:0:0]$, $V(\sigma_1)=[0:1:0]$, and  $V(\sigma_2)=[0:0:1]$. Suppose that $U_{\sigma_0}\cong \C[x_1,x_2]$ where $x_i=z_i/z_0$, $i=1,2$. Let 
$A= \left(\begin{smallmatrix} 1 & 3 \\ 2 & 2 \end{smallmatrix}\right)$, then $\varphi_A$ is the monomial map $(x_1,x_2)\mapsto (x_1 x_2^3, x_1^2 x_2^2)$. The following is the directed graph of $A$.

\includegraphics[scale=0.8]{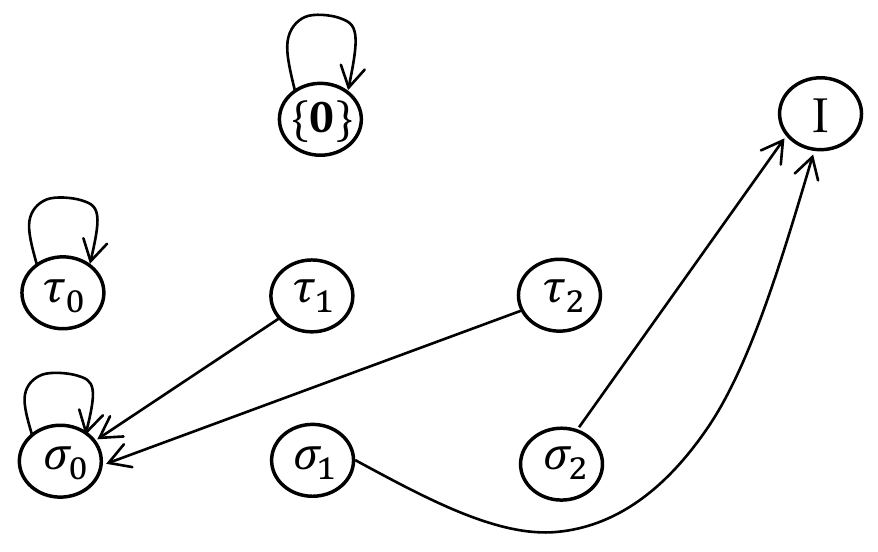}

One sees that $\widetilde{\fan}_A= \bigl\{ \{{\bf 0}\}, \tau_0, \tau_1, \tau_2, \sigma_0 \bigr\}$ and $X(\fan)_{\varphi_A}= \P^2\backslash \bigl\{[0:1:0], [0:0:1]\bigr\}$.

\end{ex}

\begin{ex} 
With the same notations as the previous example, now let $B=
\left(\begin{smallmatrix} -2 & 1 \\ 1 & 3 \end{smallmatrix}\right)$, then the directed graph associated to $B$ is the following.

\includegraphics[scale=0.8]{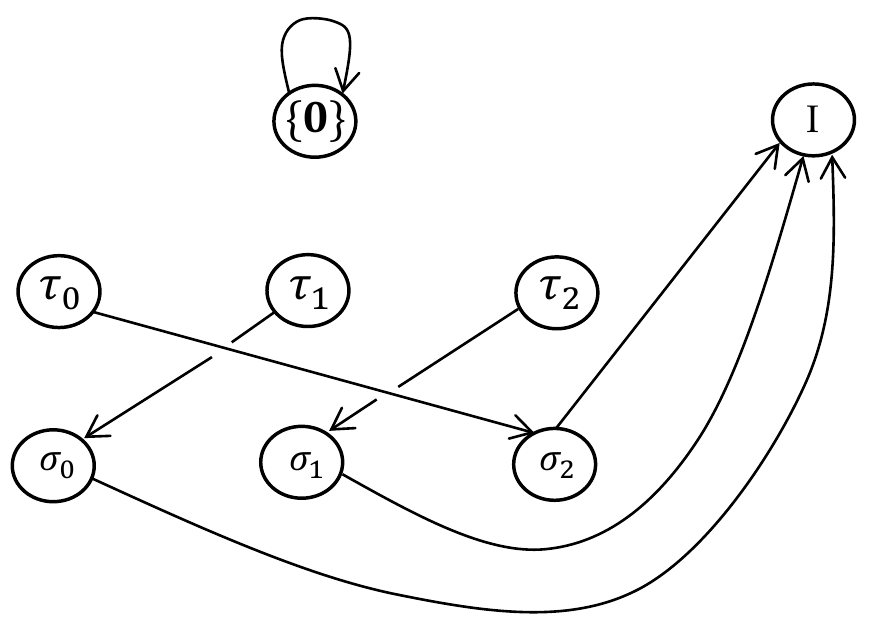}

In this case, we have $\widetilde{\fan}_B= \bigl\{ \{{\bf 0}\} \bigr\}$ and $X(\fan)_{\varphi_B}= (\C^*)^2$.

\end{ex}

Finally, let us make some remarks here. First, if we order the vertices of the graph by the dimension of the cones and give `I' the order of $+\infty$ by convention, then $\det(A)\ne 0$ implies that cones never collapse into lower dimensional cones. Hence, the arrows of the graph will be ``order preserving''. In particular, the positive dimensional cones will never map to $\{{\bf 0}\}$. This means points in $X(\fan)_{\varphi_A}$ never have a Zariski dense $\varphi_A$-orbit. Therefore the Conjecture~\ref{main_conjectures} (4) is true on $X(\fan)_{\varphi_A}$.

Next, by the pigeonhole principle and the fact that the number of cones is finite, we can easily conclude the following, which is not true in general for rational maps.

\begin{cor}
Let $\varphi:X\dashrightarrow X$ be a monomial map on a toric variety. There exists an integer $N>0$, depending on $X$, such that for all $P\in X$, if $\varphi(P), \varphi^2(P),\cdots, \varphi^N(P)$ are well-defined, then $P\in X_\varphi$.\hfill\qed
\end{cor}

See \cite[Proposition 4.11]{HP} for a similar proposition.

\subsection{The dynamical Mordell-Lang conjecture for monomial maps}
\label{section:DML}
As the first application of our description of the orbits under monomial maps, we will prove the dynamical Mordell-Lang conjecture for monomial maps on arbitrary toric varieties.

\begin{thmB}
Let $X=X(\fan)$ be a toric variety, $\varphi=\varphi_A$ be a monomial map, and $Y\subset X$ be a closed subvariety of $X$. Then for all $P\in X_\varphi$, the sequence
\[
\{ k\in \N\ |\ \varphi^k(P)\in Y\}
\]
 is a finite union of arithmetic progressions.
\end{thmB}

\begin{proof}
The dynamical Mordell-Lang conjecture is known for \'{e}tale endomorphisms of quasiprojective varieties~\cite{BGT}.  In particular, monomial maps on the algebraic torus $\varphi_A:\T\to\T$ are \'{e}tale, thus the dynamical Mordell-Lang conjecture holds for $\xx\in\T$.

For an arbitrary $\xx\in X_\varphi$, from the description of $X_\phi$ using the directed graph associated to $\varphi=\varphi_A$, we know that the points $\varphi^k(\xx)$ will eventually fall periodically into a finite number of torus orbits. More precisely, there exist $\sigma_1,\cdots, \sigma_p\in\fan$ such that $\varphi^{pk+j}(\xx)\in O_j := O_{\sigma_j}$ for $k$ large enough and $1\le j\le p$.  

Thus, for a closed subvariety $Y\subset X$, we consider $Y_j=Y\cap O_j$, which is a closed subvariety of $O_j$. Since the iterations of $\xx$ eventually fall periodically on $O_1,\cdots,O_p$, it suffices to show that for each $j$, the sequence
\[
\{ \ell \ |\ \varphi^\ell(\xx)\in Y_j\}
\] 
is a finite union of arithmetic sequences. However, since $O_j$ is itself an algrbraic torus, and $\varphi^p|_{O_j}$ is monomial, we know the dynamical Mordell-Lang conjecture holds for $\varphi^p|_{O_j}$. Therefore, the above sequence is indeed a finite union of arithmetic sequences, and the proof is complete.
\end{proof}


\section{Height Growth Under Monomial Maps}
\label{section:height_growth}

We focus on the study of height growth in this section. The first goal is to prove Theorem~\ref{thm:main_1}. Before we do so, we will need two lemmas. The first lemma states that for a point in the algebraic torus, the height growth under a monomial map only depend on the monomial map and is independent of the compactification (i.e., the toric variety  which contains the torus as a dense open subset). 

\begin{lem}
\label{lem_indep_of_embed}
Let $X_1, X_2$ be two $n$-dimensional projective toric varieties. Thus,  they both contain the $n$-torus $\T$ as a dense open subset. Denote $h_i=h_{X_i}$, the corresponding height functions on $X_i$, $i=1,2$.
Then we have $h_1(\xx)\asymp h_2(\xx)$ for $\xx\in\G_\m^n$. More precisely, there exist $C_1, C_2 >0$ such that for all $\xx\in\G_\m^n$, we have $C_1 h_2(\xx)+ O(1) \le h_1(\xx) \le C_2 h_2(\xx) + O(1)$.  
\end{lem}

\begin{proof}
First of all, choosing different ample divisors, say on $X_1$, will result in height functions that are asymptotically equivalent~\cite[Proposition 12]{KaSi}. Thus, without loss of generality, we may choose $\T$-invariant very ample divisors $D_1, D_2$ on $X_1, X_2$, respectively, to compute the heights. Such divisors correspond to polytopes $\PP_1,\PP_2\subset M_\R$. 

The embeddings of $X_i$ into projective spaces are given by the characters $\PP_i\cap M$. 
$D_1, D_2$ being very ample implies that $\PP_1, \PP_2$ are of full dimension. As a consequence, there is a positive integer $C$ and an element $v\in M$ such that $\PP_1\subset C\PP_2 + v$. This implies
\[
h_{X_1,D_1}\le h_{X_2, C D_2} +O(1) = C h_{X_2, D_2} +O(1).
\]
The other inequality can be proved similarly. Therefore, we have $h_1\asymp h_2$.
\end{proof}

The second lemma is about linear algebra concerning the spectral radius of a linear map on a subspace invariant under iterates of the map. We will need it when we look at the monomial map on various orbits of a toric variety. I would like to thank Jeff Diller for discussing the proof of this lemma with me.

\begin{lem}
\label{lemma:spectral_radius}
Let $T:\R^n\to\R^n$ be a linear transformation, $W$ be a subspace of $\R^n$ which is invariant under $T^p$ but not fewer iterates of $T$, and $V$ be the smallest subspace of $\R^n$ which contains $W$ and is invariant under $T$, i.e., $V$ is the span of $W, T(W), T^2(W),\cdots, T^{p-1}(W)$. Then we have 
\[
\rho(T^p|_W)=\rho(T|_V)^p.
\]
\end{lem}

\begin{proof}
First, since both $V$ and $W$ are invariant under $T^p$, we have 
\[
\rho(T^p|_W)\le \rho(T^p|_V)=\rho((T|_V)^p)=\rho(T|_V)^p.
\]

We will need to prove the inequality in the other direction. Suppose that $v\in V$ is an eigenvector which realize the spectral radius $\rho(T|_V)$, i.e., $T(v)=\lambda v$ and $|\lambda|=\rho(T|_V)$. Write $v=w_0 + \cdots + w_{p-1}$ such that $w_j\in W_j:=T^{j}(W)$. Notice that each $W_j$ is invariant under $T^p$. Moreover, each $T^p|_{W_j}$ is conjugate to  $T^p|_{W}$ by $T^j$, thus there are positive numbers $C_0=1, C_1, \cdots, C_{p-1}$ such that 
\[
\|T^p|_{W_j}\|\le C_j\cdot\|T^p|_{W}\|.
\]
Here $\|T^p|_{W_j}\|$ is the operator norm with respect to some norm $\|\cdot\|$ on $\R^n$.
Taking the $pk$-th iterate of $T$, we have
\begin{align*}
\rho(T|_V)^{pk}\| v \| &=  \| T^{pk}(v) \|  \\
             & \le \sum_{j=0}^{p-1} \|T^{pk}(w_j)\| \\
             & \le \sum_{j=0}^{p-1} \|(T^{p}|_{W_j})^{k}\| \cdot \| w_j\|\\
             & \le \sum_{j=0}^{p-1} C_j\|(T^{p}|_{W})^{k}\| \cdot \| w_j\| \\
             & = \|(T^{p}|_{W})^{k}\| \cdot \sum_{j=0}^{p-1} C_j\| w_j\|             
\end{align*}
Taking $k$-th root on both sides, we get 
\[
\rho(T|_V)^{p}\le C^{1/k} \|(T^{p}|_W)^k\|^{1/k}
\]
where $C>0$ is a constant independent of $k$. Finally, take the limit $k\to\infty$, we obtain the desired inequality
\[
\rho(T|_V)^{p}\le \lim_{k\to\infty}C^{1/k} \|(T^{p}|_W)^k\|^{1/k} = \rho(T^{p}|_W).
\]
This completes the proof of the lemma.
\end{proof}

Now we are ready to prove Theorem~A. Let us recall the statement of the theorem first.

\begin{thmA}
Let $\varphi:X\dashrightarrow X$ be a monomial map defined on a projective toric variety $X=X(\fan)$ by a matrix $A\in\Mat^+(\Z)$. Suppose that $f_1(t),\cdots, f_s(t)\in\Z[t]$ be the monic irreducible factors of the characteristic polynomial $\det(A-t I)$ of $A$. Then $\alpha_\varphi(P)$ exist for all $P\in X(\Qbar)_\varphi$, and
\[
\{\alpha_\varphi(P)\ |\ P\in X(\Qbar)_\varphi\} = \{ 1, \rho(f_1),\cdots,\rho(f_s)\}.
\]
\end{thmA}

\begin{proof}
For points in $\G_m^n(\Qbar)$, it is proved in \cite[Corollary 32]{Sil} that the set of the arithmetic degrees is exactly $\{ 1, \rho(f_1),\cdots,\rho(f_s)\}$. The height function used in the proof is the one from the embedding $\G_m^n(\Qbar)\subset \P^n$. By Lemma~\ref{lem_indep_of_embed}, it will give us the same arithmetic degrees as the embedding $\G_m^n(\Qbar)\subset X$ we use here. 

Thus, it remains to show that for $P\in X_\varphi(\Qbar)\setminus \G_\m^n(\Qbar)$, $\alpha_\varphi(P)$ exists and is in the set $\{ 1, \rho(f_1),\cdots,\rho(f_s)\}$. From the discussion in Section~\ref{section:orbits_under_monomial_maps}, we know that a point $P\in X_\varphi(\Qbar)$ must be in an orbit which is ``preperiodic'' (in the graph associated to $\varphi$. By \cite[Lemma 15]{KaSi}, it suffices to prove for points on a periodic orbit. Moreover, when the orbits become periodic, every orbit in a period has the same dimension. Hence, we can assume we are in the situation that $P\in O_0$ and $\varphi^l(P)\in O_{l}$ where $O_l=O_{l'}$ if $l\equiv l'\mod p$ for some $p>0$. Here, we use $O_l:=O_{\sigma_l}$ to denote $\T$-orbits of the same dimension.

To find the arithmetic degree of $P$ under $\varphi$, we need to look at the limit
\[
\lim_{k\to\infty}h_X(\varphi^k(P))^{1/k}.
\]
We know the subsequence $h_X(\varphi^{kp+l}(P))^{1/(kp+l)}$ converge to  $\alpha_{\varphi^p}(\varphi^l(P))$, since $\varphi^p$ is a monomial selfmap on $O_l$. Thus the sequence has $p$ limit points, and we need to show these values are in fact the same.

Without loss of generality, we consider the subsequences $h_X(\varphi^{kp}(P))^{1/(kp)}$ and $h_X(\varphi^{kp+1}(P))^{1/(kp+1)}$, which converge to $\alpha_{\varphi^p}(P)$ and $\alpha_{\varphi^p}(\varphi(P))$, respectively. The mappings can be illustrated by the following diagram, and we need to show that $\alpha_{\varphi^p}(P)=\alpha_{\varphi^p}(\varphi(P))$.
\[
\xymatrix{
P\in  O_0\ar[d]_{\varphi}\ar[rr]^{\varphi^p} & & O_0\ar[d]^{\varphi}\\
\varphi(P)\in  O_1\ar[rr]_{\varphi^p} & & O_1
}
\]
In fact, a claim which is slightly more general than the above is still true.
\begin{claim}
Let $\T$ be an algebraic torus and in the following diagram, $\varphi_A$, $\varphi_B$, and $\varphi_C$ are all dominant monomial maps $\T\to\T$. 
\[
\xymatrix{
\T\ar[d]_{\varphi_C}\ar[rr]^{\varphi_A} & & \T\ar[d]^{\varphi_C}\\
\T\ar[rr]_{\varphi_B} & & \T
}
\]
Suppose the diagram commutes, i.e., $\varphi_A$ and $\varphi_B$ are semiconjugate by $\varphi_C$. Then for $P\in\T(\Qbar)$ and $Q=\varphi_C(P)$, we have $\alpha_{\varphi_A}(P)=\alpha_{\varphi_B}(Q)$.
\end{claim}

\begin{proof}[Proof of Claim]\let\qed\relax
Observe that by \cite[Corollary 32]{Sil}, $\alpha_{\varphi_A}(P)$ can be calculated in the following way. First, find the smallest algebraic subgroup $G$ that contains the $\varphi_A$-orbit of $P$. Then the identity component $G_0$ of $G$ is invariant under $\varphi_A$, and $\alpha_{\varphi_A}(P)=\delta(\varphi_A|_{G_0})$.

Then, since $\varphi_A$, $\varphi_B$ and $\varphi_C$ are group homomorphisms, and $\varphi_B\circ \varphi_C = \varphi_C \circ \varphi_A$, one can show that $\varphi_C(G)$ is the smallest algebraic subgroup of $\T$ containing the $\varphi_B$-orbit of $Q$. In particular, $\varphi_C$ is a monomial map that maps $G_0$ onto $\varphi_C(G)_0$, the identity component of $\varphi_C(G)$. Furthermore, $G_0$ and $\varphi_C(G)_0$ are invariant under $\varphi_A$ and $\varphi_B$, respectively. Hence we have the semiconjugation on the restrictions
\[
\xymatrix{
G_0\ar[d]_{\varphi_C|_{G_0}}\ar[rr]^{\varphi_A|_{G_0}} & & G_0\ar[d]^{\varphi_C|_{G_0}}\\
\varphi_C(G)_0\ar[rr]_{\varphi_B|_{\varphi_C(G)_0}} & & \varphi_C(G)_0
}
\]
Suppose that $\varphi_A|_{G_0}$, $\varphi_B|_{\varphi_C(G)_0}$ and $\varphi_C|_{G_0}$ are monomial maps associated with integer matrices $A_0$, $B_0$, and $C_0$, respectively. Then we have $C_0 A_0 = B_0 C_0$. Over $\Q$, we have the conjugation $A_0 = C_0^{-1} B_0 C_0$, which implies that $A_0$ and $B_0$ has the same set of eigenvalues, hence the same spectral radius. This show that $\varphi_A|_{G_0}$ and $\varphi_B|_{\varphi_C(G)_0}$ have the same dynamical degree, and therefore $\alpha_{\varphi_A}(P)=\alpha_{\varphi_B}(Q)$.\hfill$\blacksquare$
\end{proof}

Hence, we know that the sequence defining the arithmetic degree of $P$ under $\varphi$ converges, and the limit is $\alpha_{\varphi^p}(P)^{1/p}$, i.e., $\alpha_{\varphi}(P)=\alpha_{\varphi^p}(P)^{1/p}$. On the other hand, $\varphi^p$ is associated with the matrix $A^p$, and 
$\varphi^p|_{O_0}$ is associated with the map $\bar{A^p}:N(\sigma_0):=N/N_{\sigma_0}\to N(\sigma_0)$ induced from $A^p$. Notice that $N_{\sigma_0}$ is invariant under $A^p$, not $A$. Applying lemma~\ref{lemma:spectral_radius}, we know that $\alpha_{\varphi^p}(P)$ is equal to $\rho(f_j)^p$ for some $j$. Therefore, 
\[ 
\alpha_{\varphi}(P)=\alpha_{\varphi^p}(P)^{1/p}=\rho(f_j)
\]
and the proof of the theorem is complete.
\end{proof}

Next, we will answer two questions of Silverman. The first question is \cite[Question 13]{Sil}. It asks for a description of the set
\[
\{ Q\in\P^n(\Qbar)_\varphi\ :\ \alpha_\varphi(Q)=\alpha_\varphi(P) \}.
\]
We will describe a slightly different set for monomial maps. Let $\varphi=\varphi_A:X\dashrightarrow X$ be a monomial map. Order the possible arithmetic degrees as
\[
1=\rho_0 < \rho_1 < \cdots < \rho_s.
\]
We will describe the set
\[
Z_j := \{ Q\in \G^n_\m(\Qbar)\ :\ \alpha_\varphi(Q)\le \rho_j \},
\]
for $j=0,1,\cdots,s$.

Let $\frak L$ be the set of sub-lattices of $M$ that are invariant under $\transp{A}$, i.e.,
\[
{\frak L} = \{ L\subset M\ :\ \text{$L$ is a saturated sublattice and $\transp{A}L\subset L$} \}.
\]
For $L\in {\frak L}$, $\transp{A}$ induces a linear selfmap 
\[
\overline{\transp{A}}: M/L \longrightarrow M/L.
\]
Let $r(L)$ denote the spectral radius of $\overline{\transp{A}}$.

\begin{prop}
\[
Z_j = \bigcup_{L\in{\frak L},\, r(L)\le\rho_j} \G_L^\div
\]
\end{prop}

\begin{proof}
Suppose $Q\in \G_L^\div$ for some $L$ with $r(L)\le \rho_j$, then the smallest algebraic subgroup containing the $\varphi$-orbit of $Q$ is a subgroup of $\G_L^\div$. Thus $\alpha_\varphi(Q)=\delta(\varphi|_{\G_L})$. By definition, $L$ is the subgroup of $M$ consisting of the characters which vanish (i.e., $=1$) on $\G_L$. Hence, the character lattice of $\G_L$ is the quotient lattice $M/L$ and the monomial map $\varphi|_{\G_L}$ corresponds to the linear map induced by $\transp{A}$, that is, $\overline{\transp{A}}$. Therefore, 
\[
\alpha_\varphi(Q)=\delta(\varphi|_{\G_L})=\text{(spectral radius of $\overline{\transp{A}}$)} < \rho_j,
\]
which implies $Q\in Z_j$.

Conversely, for $Q\in Z_j$, suppose the smallest algebraic subgroup which contains the $\varphi$-orbit of $Q$ is $\G_L$ for a sublattice $L\subset M$. Then $\G_L$ is invariant under $\varphi$ implies that $L$ is invariant under $\transp{A}$. Then the saturation $\tilde{L}$ is a lattice in $\frak L$ and $r(\tilde{L})<\rho_j$. This completes the proof.
\end{proof}

The second question we are going to answer is \cite[Question 15]{Sil}, and we answer by providing a counterexample. Recall that, for a rational map $\varphi: \P^n \dashrightarrow\P^n$, assuming the conjecture that $\deg(\varphi^k)\asymp k^\ell \delta(\varphi)^k$ for some integer $\ell\ge 0$  (the conjecture is true for monomial maps), Silverman defined the canonical height $\hat{h}_\varphi(P)$ for $P\in\P^n(\Qbar)$ by
\[
\hat{h}_\varphi(P)=\limsup_{k\to\infty} \frac{h(\varphi^k(P))}{k^\ell\delta(\varphi)^k}.
\]
He then asked the question whether the sequence $\left\{ \frac{h(\varphi^k(P))}{k^\ell \delta(\varphi)^k} \right\}_{k=1}^{\infty}$ in the limit has only finitely many accumulation points. The following example shows a rational map where the sequence accumulates on a whole interval.

Following \cite{Sil}, we make the following definitions to simplify notations. For a number field $\kappa$, $P=(x_1,\cdots,x_n)\in\kappa^n$, and $v\in M_\kappa$, we define $\log\norm{Q}_v$ to be the column vector 
\[
\log\norm{Q}_v:= \transp{\bigl(\log\norm{x_1}_v,\cdots, \log\norm{x_n}_v\bigr)}.
\]
by defining this notation, we have the following equation
\[
\log\norm{\varphi_A(P)}_v = A \log\norm{P}_v.
\]
Also, for a real vector $\mathbf{a}=(a_1,\cdots,a_n)$, we define
$\max^+(\mathbf{a}):=\max\{0,a_1,\cdots,a_n\}$.
 
\begin{ex}
Pick any $2\times 2$ integer matrix $A$ such that the eigenvalues of $A$ is a pair of conjugate complex numbers $\mu,\bar{\mu}$ such that $\mu/\bar{\mu}$ is not a root of unity (i.e. the argument of $\mu$ is not a rational multiple of $2\pi$). 
Such matrix $A$ is similar to a real matrix of the form $\delta\cdot
\left(\begin{smallmatrix}
\cos\theta & -\sin\theta \\ \sin\theta & \cos\theta
\end{smallmatrix}\right).
$ where $\delta$ is the modulus of $\mu$ (hence is also the dynamical degree of $\varphi_A$) and $\theta$ is the argument of $\mu$.

For a point $P=(x,y)\in (\bar{\Q}^*)^2$, in order to compute its canonical height, we first fix a number field $\K$ containing $x,y$. Next, we consider the sequence $\{\frac{1}{\delta^i} h(\phi_A^i(P))\}_{i=0}^\infty$ in the definition of canonical height. We have
\[
\frac{1}{\delta^k} h(\varphi_A^k(P))
=\frac{1}{\delta^k\cdot [\K:\Q]}\sum_{v\in M_\K} \maxp (A^k \log\norm{P}_v)
=\frac{1}{[\K:\Q]}\sum_{j=1}^l \maxp \left(
(\frac{A}{\delta})^k 
\log\norm{P}_j\right).
\]
The first equality is by definition. For the second equality, note that for each point $P$, the sum in the height definition is in fact finite, i.e. $\log\norm{P}_v\ne 0$ for only finitely many $v\in M_\K$. We use $\log\norm{P}_j$, $j=1,\cdots, l$ to denote those values.

We define $R_\phi:= 
\left(\begin{smallmatrix}
\cos\phi & -\sin\phi \\ \sin\phi & \cos\phi
\end{smallmatrix}\right)
$
to be the rotation matrix with angle $\phi$.
Notice that $A/\delta$ is then similar to the matrix $R_\theta$, 
say $A/\delta = B R_\theta B^{-1}$ for some invertible matrix $B$.
Moreover, we also know that $(A/\delta)^k$ is similar to $R_{k\theta}$.

\begin{lem}
Let $\theta$ be an angle such that $\theta/\pi\not\in\Q$, and $\v_j=(x_j,y_j)$, $j=1,\cdots,l$, be finitely many nonzero vectors in $\R^2$. Treating $(\v_1,\cdots,\v_l)$ as a point in $\R^{2l}$, then for any continuous function $f:\R^{2l}\to\R$, the closure of the set
\[
\bigl\{
f(R_{k\theta}\v_1,\cdots, R_{k\theta}\v_l)
\bigr\}_{k=0}^\infty
\]	
is either a point, or a close interval in $\R$.	
\end{lem}

\begin{proof}
Since $\theta/\pi\not\in\Q$, we know that for $\v_1$, the set
\[
S=\overline{\{R_{k\theta}\v_1\}_{k=0}^\infty } 
= \{ R_{\phi}\v_1\}_{\phi\in\R/(2\pi\Z)}
\]
is (topologically) a circle. Moreover, once we know a point $p\in S$, then we can find the corresponding $\phi$. This observation implies that the closure of the set 
\[
\bigl\{
(R_{k\theta}\v_1,\cdots, R_{k\theta}\v_l)
\bigr\}_{k=0}^\infty
\]
is also homeomorphic to a circle. Since $f$ is continuous, the image of a circle is either a point or a closed interval in $\R$.
\end{proof}

Now, since $(A/\delta)$ is similar to $R_\theta$, the same conclusion will hold for $(A/\delta)$, with $\v_j=\log\norm{P}_j$ and $f$ be the sum of the $\max^+$ function. For a point $P$ such that the function 
\[
g(\theta):= f(B R_\theta B^{-1}\v_1,\cdots, B R_\theta B^{-1}\v_l)= \sum_{j=1}^l \maxp \left(
B R_\theta B^{-1}
\log\norm{P}_j\right)
\]
is not a constant function on the circle, the image of $f$ is a close interval. This means the original sequence $\{\frac{1}{\delta^k} h(\phi_A^k(P))\}_{k=0}^\infty$ has the whole interval as its limit points.
\end{ex}



\section{Preperiodic points of monomial maps}
\label{section:preperiodic_points}
In this subsection we describe the set of preperiodic points for an arbitrary monomial map. The major case is already described in~\cite{Sil}.

\subsection{Preperiodic points on algebraic tori}
\label{section:preper_tori}

We first describe the preperiodic points of monomial maps on an algebraic torus. First we need a lemma. Notice that the first part of the lemma is already proved in \cite[Lemma 22]{Sil}.

\begin{lem}
\label{lemma:kernel_mono_maps}
Let $A=(a_{ij})\in\Mat^+_n(\Z)$ be a matrix and let the set
\[
Z=\{Q\in\T\ |\ \varphi_A(Q)=\one \}.
\]
\begin{enumerate}
\item If $D=\det(A)\ne 0$, and $P=(x_1,\cdots,x_n)\in Z$, then each $x_i$ is a $D$-th root of unity.

\item If $D=0$, then $Z$ is the algebraic subgroup $\G_L$ of $\T$ corresponding to the lattice $L\subset M$ generated by the row vectors of $A$. In particular, $\dim(Z)=n-\rank(L)=n-\rank(A)$.
\end{enumerate}
\end{lem}

\begin{proof} 
For part (1), let $\ad(A)$ be the classical 
adjoint of $A$, then 
\[
\ad(A)\cdot A = D\cdot I_n.
\]
Now suppose $P\in Z$, apply $\varphi_{\ad(A)}$ on both sides of the equation $\varphi_{A}(P)=\one$, we obtain
\[
\varphi_{\ad(A)}(\varphi_{A}(P))=\varphi_{\ad(A)}(\one) = \one.
\]
Furthermore, we calculate the left side of the equation
\[
\begin{split}
\varphi_{\ad(A)}(\varphi_{A}(P)) &= \varphi_{\ad(A)}\circ \varphi_{A}(P)\\
                                 &= \varphi_{\ad(A)\cdot A}(x_1,\cdots,x_n) \\
                                 &= \varphi_{D\cdot I}(x_1,\cdots,x_n) \\
                                 &= (x_1^D,\cdots,x_n^D)= \one.
\end{split}
\]
Therefore, each $x_i$ is a $D$-th root of unity. 

For (2), let $u_i = (a_{ij})$ be the $i$-th row vector of $A$, then by definition,
\[
Z= \bigcap_{i=1}^n \{\chi^{u_i}=1\}=\G_L. 
\]
\end{proof}

\begin{remark}
Notice that the converse of (1) in Lemma~\ref{lemma:kernel_mono_maps} is not true. For example, let $A=\left(\begin{smallmatrix} 2 & 0 \\ 0 & 3\end{smallmatrix}\right)$ then $D=6$, but for a primitive sixth root of unity $w$, $\varphi_A(w,w)\ne\one$. 
\end{remark}

\begin{prop}
\label{prop:preper_points_mono_maps}
Let $A=(a_{ij})\in\Mat^+_n(\Z)$ be a matrix and $\varphi_A$ be the associated monomial map.
\begin{enumerate}
\item If none of the eigenvalues of $A$ is a root of unity, then $\PrePer(\varphi_A)=\T_{\tor}$.
\item If some eigenvalue of $A$ is a root of unity, then there is a nontrivial algebraic subgroup $G\subset\T$ such that $\PrePer(\varphi_A)=G^{\div}$.
\end{enumerate}
\end{prop}

\begin{remark}
Notice that $G^{\div}=G\cdot\T_{\tor}$, thus (1) is in fact the special case of (2) for $G=\{\one\}$, the trivial subgroup.
\end{remark}

\begin{proof}
Part (1) is proved in \cite[Proposition 21 (d)]{Sil}. To prove (2), we will first explain how to construct the group $G$ in next two paragraphs.

Suppose $A\in \Mat^+_n(\Z)$ is a matrix which has some roots of unity as its eigenvalues, then for some $k\ge 1$ the matrix $A^k-I$ is singular, i.e., $\rank(A^k-I)<n$. We look at the Jordan form of $A$. Let $B$ be a single Jordan block of $A$ of an eigenvalue $w$ which is a primitive $k$-th root of unity. That is, $B$ is of the form
\[
B=\begin{pmatrix}
w & 1 & 0 & \cdots & 0 \\ 
0 & w & 1 & \cdots & 0 \\
\vdots & \vdots & \vdots& \ddots & \vdots\\
0 & 0 & 0 & \cdots & 1 \\
0 & 0 & 0 & \cdots & w \\
\end{pmatrix}
\]
Then for every positive integer $l$, we have 
\[
B^{kl} - I =\begin{pmatrix}
0 & kl & * & \cdots & * \\ 
0 & 0 & kl & \cdots & * \\
\vdots & \vdots & \vdots& \ddots & \vdots\\
0 & 0 & 0 & \cdots & kl \\
0 & 0 & 0 & \cdots & 0 \\
\end{pmatrix}
\]
Thus, the rank of $B^{kl}-I$ is its size minus one. Moreover, for any integer $s\ge 0$, $B^s(B^{kl}-I)$ has the same rank as $B^{kl}-I$; and for any other Jordan block whose corresponding eigenvalue is not a $kl$-th root of unity, $B^s(B^{kl}-I)$ will have full rank.

Now we observe the matrix $A$. Write $A$ into the Jordan form as $A=PJP^{-1}$. Apply the argument in the previous paragraph to Jordan blocks in $J$, we see that for some $k_0$, the rank $\rank(A^{k_0}-I)$ reaches its minimum. More precisely, let $\ell$ denote the number of Jordan blocks whose associated eigenvalues are roots of unity, then 
\[
\rank(A^{k_0}-I)= n - \ell.
\]
Let $L(A):= M\cap \row(A^{k_0}-I)$, i.e. the intersection of $M$ and the row space of $A^{k_0}-I$. We claim that
\[
\PrePer(\varphi_A) = \G_{L(A)}^{\div},
\]
i.e., the group $G$ in the proposition is $\G_{L(A)}$.

Suppose that $P\in \PrePer(\varphi_A)$, then $\varphi_{A^{r}}(P)=\varphi_{A^{s}}(P)$ for some $r > s \ge 0$. This implies that 
$\varphi_{A^{r}-A^s}(P)=\varphi_{A^{s}(A^{r-s}-I)}(P)=\one$. If 1 is not an eigenvalue of $A^{r-s}$, then by part (1) of Lemma~\ref{lemma:kernel_mono_maps}, each coordinate of $P$ is a root of unity, hence $P\in \G_{L(A)}^{\div}$. If 1 is an eigenvalue of $A^{r-s}$, then the row space (over $\R$) of $A^{s}(A^{r-s}-I)$ contains the row space of $A^{k_0}-I$ as a subspace (this can be seen using the Jordan form of $A$). Hence by part (2) of Lemma~\ref{lemma:kernel_mono_maps},
\[
P\in \G_{A^{r}-A^s} \subset \G_{A^{r}-A^s}^{\div}\subset \G_{A^{k_0}-I}^{\div} = \G_{L(A)}^{\div}.
\]

Conversely, suppose that $P\in \G_{L(A)}^{\div}=\G_{L(A)}\cdot \T_{\tor}$. By Lemma~\ref{lemma:kernel_mono_maps}, part (2), one can conclude that $\varphi_{A^{k_0}-I}(P)=\bmu$ for some $\bmu\in\T_{\tor}$, i.e. all coordinates of $\bmu$ are roots of unity. Expanding the equality, we get $\varphi_A^{k_0}(P)=\varphi_{A^{k_0}}(P)=\bmu P$. This implies that the orbit of $P$ is finite. Therefore $P\in\PrePer(\varphi_A)$.
\end{proof}

\subsection{Preperiodic points on toric varieties}
\label{section:preper_toric_var}
For a monomial selfmap $\varphi=\varphi_A$ on a toric variety $X=X(\fan)$, we can again use the $\T$-orbit decomposition discussed in Section~\ref{section:orbits_under_monomial_maps} to find the preperiodic points on $X_\varphi$. We know $X_\varphi$ is a toric variety, thus can also be written as the disjoint union of finitely many $\T$-orbits, and each orbit itself is an algebraic torus. There are two kinds of $\T$-orbits, as follows.

The first kind is the `periodic' orbit, i.e., an orbit $O_\sigma$ such that $\varphi^p(O_\sigma)=O_\sigma$ for some integer $p\ge 1$. For a point $P\in O_\sigma$, $P$ is preperiodic for $\varphi$ if and only if $P$ is preperiodic for $\varphi^p$, and $\varphi^p: O_\sigma\to O_\sigma$ is a monomial selfmap when we view $O_\sigma$ as an algebraic torus. Since this is the case discussed in the last section, we are done.

The second kind of orbit is ``preperiodic but not periodic''. This means an orbit $O_\sigma$ such that $\varphi^k(O_\sigma)\ne O_\sigma$ for all $k=1,2,\cdots$, but $\varphi^k(O_\sigma)$ is periodic, i.e. of the first kind, for some $k$. Let us pick the smallest such $k$ and denote $\psi=\varphi^k$ and $O_\tau=\psi(O_\sigma)$. Notice that by Proposition~\ref{prop:preper_points_mono_maps}, the set of preperiodic points of $\varphi$ in $O_\tau$ is of the form $G^\div$ for some algebraic subgroup $G$ (possibly trivial) of $O_\tau$. Moreover, a point $P\in O_\sigma$ is preperiodic if and only if $\psi(P)$ is preperiodic, i.e. $\psi(P)\in G^\div$, or $P\in \psi^{-1}(G^\div)$. Finally, since $\psi$ is a group homomorphism, one has the equality that $\psi^{-1}(G^\div)=\psi^{-1}(G)^\div$. 

We sum up our discussion as follows.

\begin{prop}
Let $\varphi:X \dashrightarrow X$ be a monomial selfmap on a toric variety $X$. If we write $X_\varphi$, which is also a toric variety, as a disjoint union of orbits $X_\varphi = \bigsqcup O_\sigma$. Then for each $O_\sigma$, $\PrePer(\varphi)\cap O_\sigma = G_\sigma^\div$ for some algebraic subgroup $G_\sigma$ of $O_\sigma$, i.e., 
\[
\PrePer(\varphi)=\bigsqcup G_\sigma^\div.
\]
\hfill\qed
\end{prop}

\subsection{The Dynamical Manin-Mumford Problem for Monomial Maps}
\label{section:DMM}
Preperiodic points of a rational map is closely related to the dynamical Manin-Mumford problem. 
We will consider the problem for a monomial map $\varphi=\varphi_A$ on an algebraic torus $\G^d_\m(\C)$ first. From the discussion in the previous section, we know that the set of preperiodic points under $\varphi$ in $\G^d_\m(\C)$ has the form $G^\div$, where $G$ is an algebraic subgroup of $\G^d_\m(\C)$. For the special case of $G=\{0\}$, $G^\div$ is in fact the set of torsion points. In this case, the dynamical Manin-Mumford problem becomes the original Manin-Mumford problem on algebraic tori, which is proved by Laurent.

\begin{thm}[Laurent]
If $Y\subset\G_\m^d$ is an irreducible variety, containing a Zariski dense set of torsion points of $\G_\m^d(\Qbar)$, then $Y$ is a torsion translate of an algebraic subgroup of $\G_\m^d$.
\end{thm}

In the general case where $\PrePer(\G^d_\m(\C))=G^\div$ for a nontrivial algebraic group $G$, since $G^\div=(G_0)^\div$ for the identity component $G_0$, we can assume $G$ is connected. 

We consider  the quotient group $\G^d_\m(\C)/G$. The quotient is in itself an algebraic torus, and the quotient map $\G^d_\m(\C)\to \G^d_\m(\C)/G$ is a monomial map. Concretely, let $L\subset M$ be the saturated sublattice corresponds to $G$, i.e., $G=\G_L$, and let $\aa_1,\cdots,\aa_{d-\ell}$ be a basis of $L$ (here $\ell=\dim G$ as a variety over $\C$), then we define the monomial map $\psi_\sigma$ as follows
\begin{align*}
\pi:\  \G_m^d(\C) &\longrightarrow\G_\m^{d-\ell}(\C) \\
        P  &\longmapsto (P^{\aa_1},\cdots,P^{\aa_{d-\ell}}).
\end{align*}
The map $\pi$ is surjective with kernel $G$. Thus $\G^d_\m(\C)/G\cong \G_\m^{d-\ell}(\C)$ and $\pi$ is the quotient map.

As a consequence of Laurent's Theorem, we will prove the following.

\begin{prop}
\label{prop:DMM_mono_on_torus}
Suppose that $Y\subset O_\sigma$ is an irreducible subvariety, then $Y$ contains a Zariski dense set of preperiodic points of $\varphi$ if and only if $\pi(Y)$ is a torsion translate of an algebraic subgroup of $\G_\m^{d-\ell}(\C)$.
\end{prop}

\begin{proof}
If $Y$ contains a Zariski dense set of preperiodic points of $\varphi$, then $\pi(Y)$ contains a dense set of torsion points of $\G_\m^{d-\ell}(\C)$. Hence by Laurent's Theorem, $\pi(Y)$ is a torsion translate of an algebraic subgroup of $\G_\m^{d-\ell}(\C)$.

Conversely, suppose that $\pi(Y)$ is a torsion translate of an algebraic subgroup of $\G_\m^{d-\ell}(\C)$, we need to show that the preperiodic points of $\varphi$ is dense in $Y$. First, notice that the set preperiodic points of $\varphi$ is the union of the fibers $\pi^{-1}(Q)$ over all torsion points $Q$ of  $\G_\m^{d-\ell}(\C)$.

For a closed subvariety $Z\subsetneq Y$ (in particular $\dim Z < \dim Y$), if $\pi(Z)\ne \pi(Y)$, then $\pi(Z)$ cannot contain all the torsion points in $\pi(Y)$. Thus, $Z$ does not contain all the preperiodic points of $Y$.

Now assume $\pi(Z)= \pi(Y)$, for $Q\in\pi(Y)$, denote $Y_Q:=\pi^{-1}(Q)\cap Y$ and $Z_Q:=\pi^{-1}(Q)\cap Z \subset Y_Q$. For a generic point $Q\in  \pi(Y)$, we have
\[
\dim Y_Q = \dim Y-\dim \pi(Y) > \dim Z-\dim \pi(Z) = \dim Z_Q.
\]
However, since the torsion points are Zariski dense in $\pi(Y)$,  there exists some torsion point $Q\in \pi(Y)$ such that $\dim Y_Q > \dim Z_Q$. Therefore, $Z$ cannot contain all the preperiodic points in $Y$. This completes the proof.
\end{proof}

For a monomial map $\varphi:X\dashrightarrow X$ on a toric variety, the restriction $\varphi|_{X_\varphi}:X_\varphi\to X_\varphi$ is a dominant morphism on the quasiprojective variety $X_\varphi$. If we write $X_\varphi$ as the disjoint union of $\T$-orbits $X_\varphi = \bigsqcup O_\sigma$, and suppose that $Y\subsetneq X$ is a closed irreducible subvariety, then $Y\cap O_\sigma$ is Zariski dense in $Y$ for a unique cone $\sigma$. Applying Proposition~\ref{prop:DMM_mono_on_torus} to the algebraic torus $O_\sigma$ and its closed subvariety $Y\cap O_\sigma$, we can conclude the following answer to the dynamical Manin-Mumford problem for monomial maps.

\begin{thmC}
Let $\varphi:X \dashrightarrow X$ be a monomial selfmap on a toric variety $X$. If we write $X_\varphi$ as a disjoint union of orbits $X_\varphi = \bigsqcup O_\sigma$, and write the set of preperiodic points of $\varphi$ as $\PrePer(\varphi)=\bigsqcup G_\sigma^\div$ for connected algebraic subgroups $G_\sigma$ of $O_\sigma$, as in Proposition~\ref{prop:preper_points_mono_maps}. 

For a closed irreducible subvariety $Y\subset X_\varphi$, let $\sigma$ be the unique cone such that $Y\cap O_\sigma$ is Zariski dense in $Y$, and let $\pi: O_\sigma\to O_\sigma / G_\sigma$ be the quotient map. Then the preperiodic points of $\varphi$ is Zariski dense in $Y$ if and only if $\pi(Y\cap O_\sigma)$ is a torsion translate of an algebraic subgroup of $O_\sigma$.
\hfill\qed
\end{thmC}



\end{document}